\documentclass{amsart}

\usepackage{hyperref}
\usepackage{lineno}
\usepackage{stackrel}
\usepackage{dsfont}
\usepackage{amsmath}
\usepackage{enumitem}
\usepackage{amsthm}
\usepackage{comment}
\usepackage{fullpage}
\usepackage{multicol}
\usepackage{xcolor}
\usepackage{mathrsfs}
\usepackage{graphicx}
\usepackage{amssymb}
\usepackage{csquotes}
\usepackage{graphics}
\usepackage{cleveref}

\newtheorem{theorem}{Theorem}
\newtheorem{lemma}[theorem]{Lemma}
\newtheorem{proposition}[theorem]{Proposition}

\newtheorem{remark}[theorem]{Remark}

\DeclareMathOperator*{\argmin}{argmin}
\newcommand{\A}{\mathbf{A}}
\newcommand{\aaa}{\mathbf{a}}

\newcommand{\E}{\mathbf{E}}
\newcommand{\ee}{\mathbf{e}}
\newcommand{\J}{\mathbf{J}}
\newcommand{\jj}{\mathbf{j}}
\newcommand{\N}{\mathbf{N}}
\newcommand{\rr}{\mathbf{x}}
\newcommand{\RR}{\mathcal{R}}
\newcommand{\R}{\mathbb{R}}

\newcommand{\uu}{\mathbf{u}}
\newcommand{\vv}{\mathbf{v}}
\newcommand{\ww}{\mathbf{w}}
\newcommand{\OA}{\mathcal{A}}
\newcommand{\OACC}{\mathcal{A}_{cc}}
\newcommand{\OACS}{\mathcal{A}_{cs}}
\newcommand{\OASC}{\mathcal{A}_{sc}}

\newcommand{\OB}{\mathcal{B}}
\newcommand{\OC}{\mathcal{C}}

\newcommand{\OH}{\mathcal{H}}
\newcommand{\OI}{\mathcal{I}}
\newcommand{\OL}{\mathcal{L}}
\newcommand{\OP}{\mathcal{P}}
\newcommand{\OS}{\mathcal{S}}
\newcommand{\HL}[1]{H_L \left( \Omega_{#1} \right)}

\newcommand{\curl}{\rm{curl}}
\newcommand{\dive}{\rm{div}}

\begin{document}

\title{Monotonicity of the Laplace Transform for Tomography in Dissipative Systems}

\author[Tamburrino, Corbo Esposito, Piscitelli]{Antonello Tamburrino$^{1,2}$, Antonio Corbo Esposito$^1$, Gianpaolo Piscitelli$^3$}

\footnotetext[1]{Dipartimento di Ingegneria Elettrica e dell'Informazione \lq\lq M. Scarano\rq\rq, Universit\`a degli Studi di Cassino e del Lazio Meridionale, Via G. Di Biasio n. 43, 03043 Cassino (FR), Italy.\\
Email: \textrm{corbo@unicas.it, antonello.tamburrino@unicas.it} (corresponding author).}

\footnotetext[2]{EUT+ Institute of Nanomaterials and Nanotechnologies-EUTINN, European University of Technology, European Union.}

\footnotetext[3]{Dipartimento di Scienze Economiche Giuridiche Informatiche e Motorie, Universit\`a degli Studi di Napoli Parthenope, Via Guglielmo Pepe, Rione Gescal, 80035 Nola (NA), Italy.\\
Email: {\rm gianpaolo.piscitelli@uniparthenope.it}.}

\setcounter{tocdepth}{1}

\begin{abstract}
This paper addresses the problem of tomography for the interior of dissipative materials, with a focus on Magnetic Induction Tomography (MIT), a proven technique for imaging the interior of conductive materials using low-frequency electromagnetic fields.

Processing MIT data is mathematically challenging because of the non-linear and ill-posed nature of the underlying inverse problem. On the other hand, the Monotonicity Principle is recognized as the basis for developing effective approaches.

In this framework, the paper presents a principle of monotonicity for the Transfer Operator in Magnetic Induction Tomography, i.e. the operator mapping the Laplace transform of the applied source onto the Laplace transform of the measured quantity. Specifically, it is proved that the Transfer Operator satisfies a Monotonicity Principle when evaluated on a proper real semi-axis of the complex plane. The description of the related (real-time) imaging method is also given.

\vspace{0.1cm}
\noindent {\it Keywords:} 
Magnetic Induction Tomography, Monotonicity Principles, Transfer Function, Inverse Problems, Imaging.\\
{\it AMS subject classifications:} 35K05, 35K10, 35R30, 78A46.

\end{abstract}

\maketitle


\section{Introduction}
Tomography is a technique used to create a 3D image of an object by taking measurements from the outside. It works by sending a signal (like an X-ray) into the object and measuring how it comes out the other side. This can be used to image various structures, such as bones, organs, and even materials.

There are two main types of tomography: hard-field and soft-field. Hard-field tomography utilizes rather energetic fields that travel in straight lines, like X-rays. This makes it easier to create an image of the interior, but it involves ionizing radiation, expensive equipment, and safety issues, among other concerns. In contrast, soft-field tomography uses less energetic fields that can be bent by the material. This makes it more versatile because no ionizing radiation is involved; however, it comes with an increase in the level of difficulty in processing the data to form an image of the interior.

The tomographic method treated in this work is Magnetic Induction Tomography (MIT), also known as Eddy Current Tomography. Magnetic Induction Tomography relies on 
the capability of low-frequency electromagnetic fields to penetrate to some extent the interior of conducting materials. In this tomographic method, a set of electrical currents that circulate in the source region $\Omega_S$ induces electrical currents in the conducting domain under imaging $\Omega_C$ (see \Cref{fig_01_system}). The spatial distribution of the electrical resistivity $\eta$ affects the reaction field measured externally to $\Omega_C$ and, therefore, starting from the reaction field, an attempt is made to reconstruct the electrical resistivity into the interior of the conducting domain.
\begin{figure}
    \centering    \includegraphics[width=0.5\linewidth]{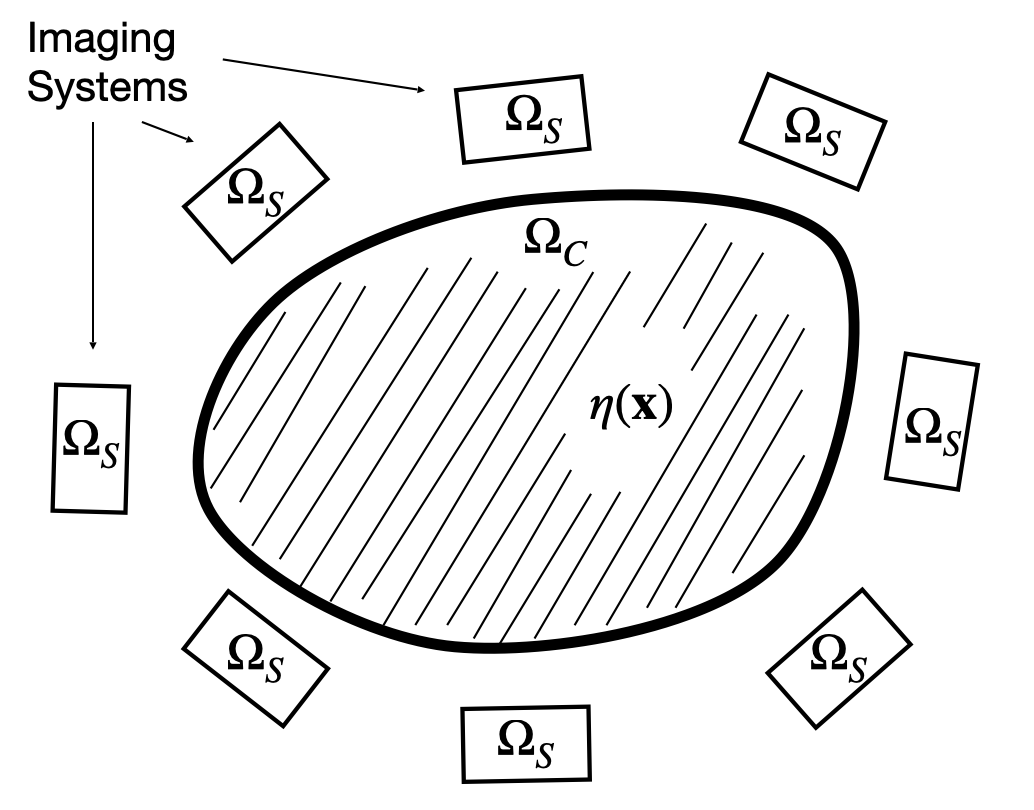}
    \caption{Schematic of a generic system for Magnetic Induction Tomography. A set of electrical currents circulating in the region $\Omega_S$ induces electrical currents in the conducting domain under imaging $\Omega_C$. The spatial distribution of the electrical resistivity $\eta$ affects the reaction field measured in the source region $\Omega_S$.}
    \label{fig_01_system}
\end{figure}

In practical scenarios where it suffices to recover only the support of an anomaly in a known background, non-iterative imaging methods are advantageous due to their real-time capabilities. The fundamental idea in this category of methods involves calculating an indicator, specifically a function of space, that takes distinct values when assessed within or outside the anomalous region. Several non-iterative imaging techniques within this context are recalled here.

Pioneering studies on non-iterative approaches are due to Colton and Kirsch (Linear Sampling Method \cite{colton1996simple}, and Factorization Method \cite{kirsch1998characterization}), originally developed within the realm of inverse scattering problems. In their papers, they characterize the support of an obstacle by considering the range of a suitable operator. Non-iterative methods include Ikeata's work on recovering the convex hull of polygonal cavities \cite{ikehata2004inverse,ikehata2011framework,ikehata2002convex} and Br\"{u}hl's extension of the Factorization Method to electrical resistance tomography \cite{bruhal2001characterization, bruhl2000locating}.
These methods rely on sophisticated mathematical constructions that require meticulous numerical implementations.
A robust numerical convergence criterion for infinite series, particularly sensitive to noise, is essential, as indicated in \cite{bruhl2000locating}.
Kusiak, Sylvester, and Potthast proposed a non-iterative method to find the so-called convex scattering support, a convex subset of the convex hull of the support of an inhomogeneity in a medium \cite{kusiak2003scattering,kusiak2005convex,potthast2003range}. 
In \cite{hanke2012one,griesmaier2014far}, Hanke, Griesmaier, and Sylvester proposed non-iterative methods for the Helmholtz equation for localizing one or several scatterers from far-field data at one fixed incident field and frequency. 
The MUSIC algorithm (MUltiple SIgnal Classification), well-established in signal processing, gained attention for its potential application in imaging, as highlighted by Devaney \cite{Devaney2000,devaney2012mathematical}. Ammari and Lesselier \cite{Ammari2005music} extended the MUSIC algorithm to locate small inclusions buried in a half-space, leveraging scattered amplitude data. Although capable of detecting multiple inclusions simultaneously, this method is limited to cases where inclusions are small and well-separated.

The Monotonicity Principle Method (MPM), which is the main subject of this paper, is a non-iterative imaging method introduced by Tamburrino and Rubinacci in 2002 for electrical resistance tomography (ERT) \cite{tamburrino2002new}.
MPM is based on a proper monotonic relation between the unknown interior material property and the measured external data. The first evidence of this monotonous relationship was found by Gisser, Isaacson, and Newell in 1990 (see \cite{gisser1990electric}).
However, its value as the foundation for a new family of imaging methods was not realized until 2002 \cite{tamburrino2002new}.

MP-based tomography presents several advantages. First, unlike other methods such as Gauss-Newton and similar (see, for instance, \cite{norton1993theory} and \cite{lionheart2003magnetic}), it does not require costly and time-consuming repetitive calculations of the forward problem, making it ideal for \emph{real-time imaging}. Additionally, MP-based algorithms can be easily broken down into smaller tasks, allowing for \emph{parallel processing} and further reduced computational time. Second, the test based on the Monotonicity Principle is \emph{exact} not only in an ideal setting consisting of a continuum of probes \cite{harrach2013monotonicity} but even for a discrete and finite distribution of probes \cite{tamburrino2002new}. Third, MP-based tomography methods are \emph{fully nonlinear}, in the sense that no approximations of the nonlinear operator mapping the unknown quantities into the measured data are introduced, unlike for the (linear) Born approximation \cite{born1926,tamburrinoventrerubinacci_2000,wang2007,zorgati1991eddy,zorgati1992eddy}. Moreover, although MP does not linearize the unknown-data operator, it does not suffer \cite{Garde2017} from the presence of \emph{local minima} and false solutions when the inverse problem is cast in terms of minimization \cite{pierri_1997,soleimani2006abscond,yusa2003corrosion}.  
Fourth, the test required by MPM can be implemented \emph{without any approximation} also in a real-world setting consisting of a finite number of noisy measurements, as opposed to other non-iterative methods, such as the Factorization Method, based on a range test that requires checking if a series is convergent or not (see \cite{bruhl2000locating, bruhal2001characterization}). Fifth, MPM can be \emph{regularized}, resulting in a stable and robust imaging method against noise (\cite{Garde2017}, \cite{mottola2025inverseobstacleproblemnonlinear}). In~\cite{corbo2021monotonicity} and~\cite{corbo2024piecewise}, the Monotonicity Principle has been introduced for nonlinear problems, under quite general assumptions on the material property. The related imaging method, together with realistic numerical examples, can be found in \cite{mottola2024imaging} and the converse in \cite{mottola2025inverseobstacleproblemnonlinear}.

MPM, initially applied in the context of ERT for elliptic problems, was later extended to parabolic partial differential equations (PDEs) governing Magnetic Induction Tomography in both the small skin-depth regime \cite{tamburrino2010recent} and the large skin-depth regime \cite{tamburrino2006fast}, with experimental validations in \cite{tamburrino2012nonite}. The extension of MPM to hyperbolic PDEs, which governs inverse wave propagation scattering problems, provided upper and lower bounds for unknown objects under the constraint of a finite number of measurements (limited aperture data) \cite{tamburrino2002new}. In \cite{harrach2013monotonicity}, it was established that MPM yields the exact shape of anomalies in the ideal scenario of an infinite number of measurements (full aperture data), assuming that each connected component of the anomalies is contractible.

MPM remains an active research topic, showing promise in various applications. Garde et al. \cite{Garde2017} formulated a monotonicity-based shape reconstruction scheme, incorporating regularization against noise and modeling error. Two reconstruction algorithms were proposed and validated using numerical examples with simulated complete electrode model (CEM) data and experimental measurements. In the medical imaging domain, a novel bimodal imaging technique was presented combining breast microwave radar (BMR) and electrical impedance tomography (EIT) methods \cite{tapia2010impedance, tapia2011bimodal}. This technique used a priori information from BMR images to estimate the location of dense breast regions and employed the monotonicity of the impedance matrix to reconstruct the tissue distribution in the breast region.

In \cite{aykroyd2005ert}, Aykroyd et al. investigated the performance of MPM in ERT problems, proposing a modified algorithm that integrates Bayesian modeling and Markov chain Monte Carlo (MCMC) estimation for configurations with a limited number of electrodes and low noise levels. Wallinger et al. \cite{wallinger2009adaptive} combined MPM with Gauss–Newton-based iterative algorithms for electrical capacitance tomography in two-phase scenarios. MPM provided the initial estimate for the Gauss–Newton method, refining the boundaries of anomalies. Harrach et al. \cite{harrach2015ultrasound} developed a monotonicity-based imaging method by combining frequency-difference EIT with ultrasound-modulated EIT, eliminating the need for numerical simulations or forward models.

The monotonicity of time constants in pulsed Magnetic Induction Tomography, initially introduced and numerically validated in \cite{tamburrino2015pect}, led to subsequent publications that developed an imaging method based on the monotonicity of time constants \cite{su2015timedomain, tamburrino2016timedomain, su2017monotonicity, tamburrino2021monotonicity}. These studies demonstrated real-time reconstruction capabilities to profile anomalies within practical 3D structures.

The mathematical framework underlying the treated problem is modeled by a parabolic system (see \Cref{sec:MatMod}), a mathematical context in which it is generally challenging to identify the correct operator that exhibits the MP. The original content of this paper consists of (i) introducing the Transfer Operator $\mathcal{H}_\eta$, a new operator to treat inverse problems in parabolic PDEs, and (ii) proving the Monotonicity Principle for $\mathcal{H}_\eta$.

The Transfer Operator $\mathcal{H}_\eta$ is the operator that maps the Laplace transform of the driving source to the Laplace transform of the measured data. Specifically, $\mathcal{H}_\eta$ evaluated for a complex $p$ is a kind of transfer \lq\lq function\rq\rq \ mapping the Laplace transform of the applied current density to the Laplace transform of the measured electrical field, as discussed in \Cref{transfer_sec}, i.e.,
\begin{equation}
\label{eq:H_defintro}
    \mathcal{H}_{\eta}(p) : 
\J_s(\rr,p) \rightarrow \OP_S \E_R(\rr,p),
\end{equation}
where $\eta$ is the electrical resistivity of the conducting domain under tomographic imaging, $\J_s(\rr,p)$ is the Laplace transform of the applied source current density evaluated at $p \in \mathbb{C}$, $\E_R(\rr,p)$ is the Laplace transform of the measured electrical field, and $\OP_S$ is a projector in a proper functional space.
The Transfer Operator represents the measured data. It generalizes the concept of generalized impedance (in the complex plane) borrowed from Circuit Theory.

The Monotonicity Property for the Transfer Operator $\mathcal{H_{\eta}}$ proved in this paper is:
\begin{equation}
\label{eq:Monotonicity_intro}
    \eta_1 \preceq \eta_2 \implies \OH_{\eta_1} \left( p \right) \preceq \OH_{\eta_2} \left( p \right), 
\end{equation}
for any $p$ real and in the region of convergence for $\OH_{\eta_1}$ and $\OH_{\eta_2}$ (see \Cref{monotransinclu} for details). In \eqref{eq:Monotonicity_intro} $\preceq$ stands for the Loewner order. The impact of a Monotonicity Property in terms of imaging methods was first recognized in \cite{tamburrino2002new} for elliptic PDEs in the context of Electrical Resistance Tomography. In this paper, the discovery of the new form of the Monotonocity Property of \eqref{eq:Monotonicity_intro} allows us to apply the imaging methods developed in the framework of MPM to the operator $\OH_{\eta}$. The foundation of the imaging method comes from the following equivalent form of \eqref{eq:Monotonicity_intro}:
\begin{equation}
\label{eq:MIP}
    \OH_{\eta_1} \left( p \right) \not \preceq \OH_{\eta_2} \left( p \right)  \implies \eta_1 \not \preceq \eta_2. 
\end{equation}
The Relation \eqref{eq:MIP} allows one to infer a relation between $\eta_1$ and $\eta_2$, starting from the measured data, i.e., the Transfer Operator evaluated at $p$.

This paper, which introduces a new operator ($\mathcal{H_{\eta}}$) to treat imaging problems for parabolic PDEs, generalizes the findings from the finite-dimensional approximation of \cite{su2023tranfer} to the continuous case. This generalization requires methods and techniques that are completely different from those previously used for the discrete setting. We refer to \cite{su2023tranfer} for an informal introduction to the proposed approach and for preliminary numerical examples that show the effectiveness of the imaging method based on the Monotonicity Property \eqref{eq:MIP} of the operator $\mathcal{H_{\eta}}$.

The results achieved in this paper are part of a broader research program dedicated to the development of Monotonicity Principles for MIT. The connection with previous studies is graphically represented in \Cref{fig_03_schema}.
\begin{figure}
\centering    \includegraphics[width=0.9\textwidth]{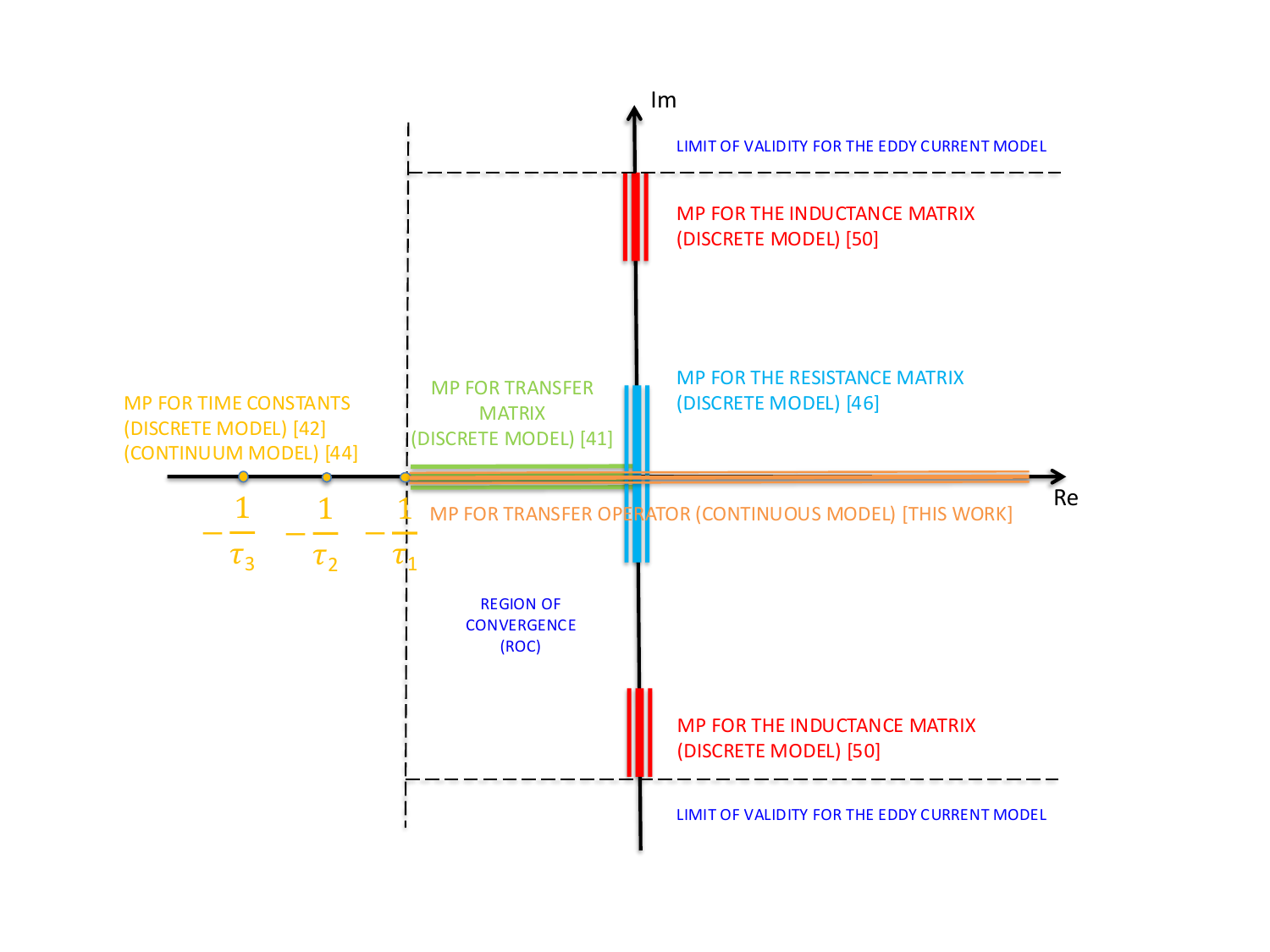}
\caption{Overview of MPs for parabolic PDEs represented in the Laplace plane. In \cite{su2023tranfer}, the data are embedded in exponentially decaying functions that may suffer from a poor signal-to-noise ratio in real-world settings. In \cite{su2017monotonicity,tamburrino2021monotonicity}, an MP for the time constants of the free response was proved. Measuring many time constants is a challenging problem. In \cite{tamburrino2006fast} a MP valid in the limit of small angular frequencies (large skin-depth) was proved. In \cite{tamburrino2010recent} a MP valid in the limit of large angular frequencies (small skin-depth) was proved. In  \cite{su2023tranfer, su2017monotonicity, tamburrino2006fast, tamburrino2010recent}, the MP was proved in a discrete setting, for both the underlying PDE and the measurement system. In this work, MP is proved for (i) quantities that that be robustly measured even in noisy environments (Transfer Operator evaluated on a right real semi-axis), (ii) for an exact model, (iii) for the continuum physical model, and (iv) for an arbitrary measurement system in infinite dimensional space.}
    \label{fig_03_schema}
\end{figure}
In \cite{su2023tranfer}, the MP for the transfer matrix due to a finite set of coils is proved in a discrete setting for sources that decay exponentially. In \cite{su2017monotonicity,tamburrino2021monotonicity}, an MP for MIT under source-free conditions is achieved (Monotonicity of the Time-constants). In contrast, the present contribution proves an MP for the forced response due to general arbitrary terms.
In \cite{tamburrino2006fast}, MP is found for the resistance matrix (real-part of the impedance matrix) measured on a set of prescribed coils (measurement system), in the low frequency limit (large skin-depth operations), while in \cite{tamburrino2010recent}, an MP is derived for the inductance matrix (imaginary-part of the impedance matrix divided by the angular frequency) at sufficiently high angular frequencies (small skin-depth operations). Except for \cite{tamburrino2021monotonicity}, all MPs provided in previous work have been proved for a finite-dimensional approximation of the underlying PDE (see \cite{su2023tranfer, su2017monotonicity, tamburrino2006fast, tamburrino2010recent}). 
The Monotonicity Property of this contribution marks a milestone in this research program. Indeed, it overcomes the limits of the Monotonicity Principles found for time-harmonic operations in \cite{tamburrino2006fast,tamburrino2010recent}, where the monotonicity was proven not for arbitrary operation frequencies but, rather, only in the small skin-depth regime and in the large skin-depth regimes, and overcomes the limits of the Monotonicity Principles for the time constants \cite{su2017monotonicity,tamburrino2021monotonicity}, where it may not be trivial to extract the time constants from the time-domain response.

The structure of the paper is as follows: in \Cref{sec:problem_statement} is described the mathematical model for Magnetic Induction Tomography; in \Cref{transfer_sec} is introduced the Transfer Operator $\mathcal{H_{\eta}}$, while in \Cref{sec:MonoTF}, its monotonicity is proved. In \Cref{method_sec}, a non-iterative imaging method based on the MP of $\mathcal{H_{\eta}}$ is described, and, eventually, in \Cref{conclusion_sec}, conclusions are drawn.

\section{Mathematical Model for Magnetic Induction Tomography}\label{sec:problem_statement}

In this article, we address the issue of determining the configuration of one or multiple anomalies located within the confines of a conductive region $\Omega_C$. This is achieved through the utilization of electromagnetic fields under magneto-quasi-stationary (MQS) conditions, where the displacement current is considered negligible \cite{bossavit1981numerical, bossavit1998computational,haus1989electromagnetic,touzani2014mathematical}. The anomalous region, denoted as $V \Subset \Omega_C$, may exhibit various topologies and shapes, potentially comprising multiple components, and possesses a resistivity distinct from that of the surrounding conducting domain $\Omega_C$. The objective of recovering the shape of the anomaly or anomalies belongs to the area of research, famously known as the inverse obstacle problem. An exemplary scenario is found in Magnetic Induction Tomography (MIT), where the focus lies on detecting "small" defects within an otherwise homogeneous material.

The MIT probing system is highly conceptual yet universally applicable, encompassing all conceivable scenarios within this broad framework. In detail, the domain under consideration, denoted as $\Omega_C$, is imaged using a set of various source current densities $\J_S$ circulating within the source domain $\Omega_S$, as shown in \Cref{fig_02_specimen}. The observed parameter is the (projected) reaction electric field $\E_R$ generated by the current density $\J$ induced within the conductive domain $\Omega_C$.

\subsection{Notations and Settings}
We consider the standard scenario where a conducting domain $\Omega_C$ is surrounded by an insulating material. Consequently, the normal component of the current density $\mathbf{J}$ evaluated on the surface of the conducting domain $\Omega_C$ is assumed to be vanishing.
\begin{figure}[htb!]
\centering
\includegraphics[width=0.5\textwidth]{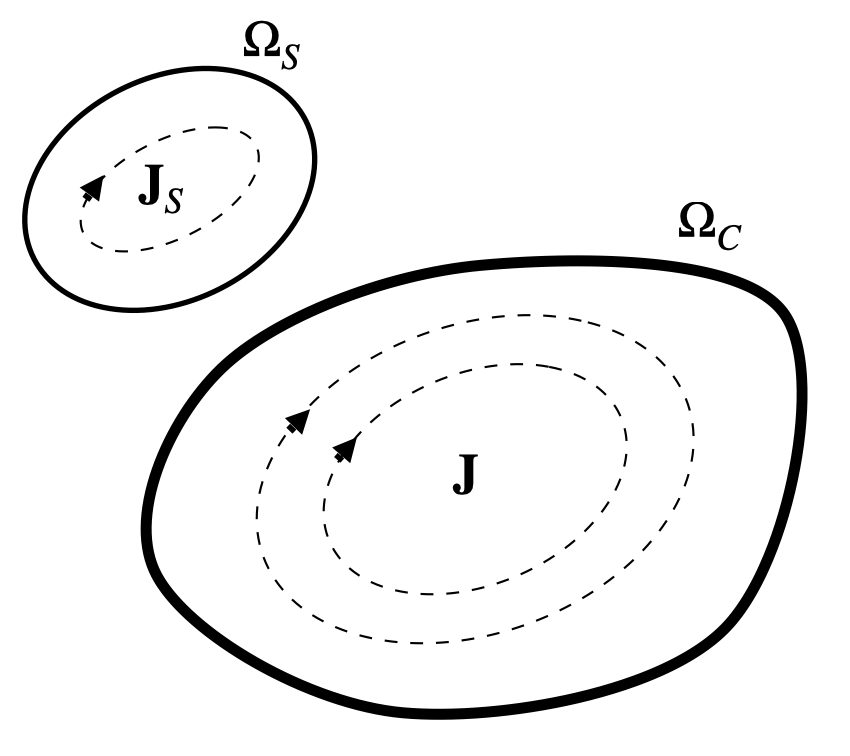}
\caption{Schematic of a typical eddy current problem: a conducting domain $\Omega_C$ and the source region $\Omega_S$. In $\Omega_S$ a source electrical current density $\J_s$ is applied and the (reaction) electric field $\E_R$ produced by the current density induced in $\Omega_C$ is measured. The operator (Transfer Operator) mapping the applied source current density $\J_s$ into $\E_R$ is the measured data.}
\label{fig_02_specimen}
\end{figure}
We recall from \cite{evans2022partial,monk2003finite} the classical functional space $H_{\rm div}(\Omega_C)$
\begin{equation}
    H_{\rm div}(\Omega_C)  =\left\{\mathbf{u}\in L^2 (\Omega_C, \mathbb R^3) \; \vert \; {\rm div} (\mathbf{u})\in L^2(\Omega_C,\mathbb R^3) \right\},
\end{equation}
equipped with the scalar product
\[
\langle\mathbf{u},\mathbf{v}\rangle_{ H_{\rm div}(\Omega_C)}=\int_{\Omega_C} \mathbf{u} (\mathbf{x}) \cdot \mathbf{v}(\mathbf{x}) \mathrm{d}\mathbf{x}+\int_{\Omega_C} {\rm div} (\mathbf{u} (\mathbf{x}) )\ {\rm div} (\mathbf{v}(\mathbf{x})) \mathrm{d}\mathbf{x},
\]
and we define the functional space $H_L\left(\Omega_C \right)$ (see \cite{bossavit1998computational} and \cite{tamburrino2021monotonicity}) as 
\begin{equation}
    H_L\left(\Omega_C \right)  = \left\{  \mathbf{v} \in H_{\rm div}(\Omega_C) \; \vert  \; \nabla \cdot \mathbf{v}=0 \text{ in } \Omega_C, \right.
    \left. \mathbf{v}\cdot \hat{\mathbf{n}} = 0 \text{ on } \partial \Omega_C \right\},  
\end{equation}
equipped with the scalar product inherited from $H_{\rm div}(\Omega_C)$. Since $H_L\left(\Omega_C \right)$ is a divergence free space, it turns out that the inherited scalar product reduces to the standard $L^2 \left( \Omega_C \right)$ inner product:
\begin{equation}
    \langle \mathbf{u},\mathbf{v} \rangle_C = \int_{\Omega_C} \mathbf{u} (\mathbf{x}) \cdot \mathbf{v}(\mathbf{x}) \mathrm{d}\mathbf{x}.
\end{equation}
Similarly, it is possible to define $H_{\rm div} (\Omega_S)$ and $H_L\left(\Omega_S \right)$.

Moreover, we set
\[
L^\infty_+(\Omega_C):=\{\theta\in L^\infty(\Omega_C)\ |\ \theta\geq c_0\ \text{a.e. in}\ \Omega_C, \ \text{for a positive constant}\ c_0\}.
\]

\subsection{Model Problem}\label{sec:MatMod}
In this section, a summary of the mathematical model for the Eddy Current problem is provided.

Let ${\bf E}$, ${\bf B}$, ${\bf H}$, and ${\bf J}$ be the electric field, the magnetic flux density, the magnetic field, and the electric current density, respectively. Under the Magneto-Quasi-Stationary approximation, Maxwell's equations reduce to (see, e.g., \cite{touzani2014mathematical}):
\begin{align}
&\label{third_max}
{\curl}\ {\mathbf{E}}(\mathbf{x},t) =-\partial_t{\bf B}(\mathbf{x},t)\qquad\qquad\textrm{in}\ \R^3\times [0,+\infty[,\\
\label{fourth_max}
&{\curl}\ {\bf H}(\mathbf{x},t) = {\bf J} (\mathbf{x},t)+{\bf J}_S(\mathbf{x},t)
\quad\ \textrm{in} \ \R^3\times [0,+\infty[ ,\\
\label{second_max}
&{\dive}\ {\bf B}(\mathbf{x},t)  = 0\qquad\qquad\qquad\qquad\ \textrm{in}\ \R^3\times [0,+\infty[,\\
\label{Ohm} 
 & {\bf J} (\mathbf{x},t) =
\begin{cases}
 \eta^{-1}(\mathbf{x}) \ {\bf E}(\mathbf{x},t) \qquad\quad\ \   \textrm{in}\ \Omega_C\times [0,+\infty[\\
 0 \qquad\qquad\qquad
 \qquad\quad\ \textrm{in} \ (\R^3\setminus\Omega_C)\times [0,+\infty[,
\end{cases}\\
\label{dpermH}
& {\bf B}(\mathbf{x},t)=\mu_0 {\bf H}(x,t)\qquad\qquad\qquad\  \textrm{in}\ \R^3\times [0,+\infty[,\\
\label{inf_van}
& |{\bf B}(\mathbf{x},t)|=O(|\mathbf{x}|^{-3})\ \textrm{uniformly for $t\in[0,+\infty[$, as} \ |\mathbf{x}| \to +\infty,
\end{align}
where $\eta \in L_+^\infty(\Omega_C)$ denotes the electrical resistivity of the conductor, $\mu_0$ is the magnetic permeability of the free space and ${\bf J}_S$ represents the prescribed source current density. Moreover, it is assumed that there are no magnetic materials in the system. Equations~(\ref{third_max}), (\ref{fourth_max}), and (\ref{second_max}) have to be interpreted in the weak sense, to automatically account for the discontinuities of the fields across material interfaces~\cite{bossavit1998computational}.

The problem described by~(\ref{third_max})--(\ref{inf_van}) admits an equivalent integral formulation (see \cite{bossavit1981numerical}). 
Specifically, the integral formulation in the weak form for the Eddy Current problem is to find $\mathbf{J}\in L^2\left(0,T;H_L\left( \Omega_C \right)\right)$ such that
\begin{equation}
\label{governing_eq}
\begin{split}
    \langle \ww,\eta \J(\cdot,t) \rangle_C  =  - \langle \ww, \partial_t \OACC \left[ \J\right(\cdot,t)] \rangle_C  - & \langle \ww , \partial_t\A_s(\cdot,t) \rangle_C,
 \end{split}
\end{equation}
for all $\ww \in H_L\left(\Omega_C\right)$ and almost all $0\leq t \leq T$.  Here $\eta$ is the point-wise resistivity, $\OACC$ is a spatial integral operator defined as
\begin{equation}
\OACC:\mathbf{v} \in H_L\left(\Omega_C \right)\rightarrow \frac{\mu_0}{4\pi}\int_{\Omega_C} \frac{\mathbf{v}(\mathbf{y})}{\|\mathbf{y}-\mathbf{x}\|}\mathrm{d}\mathbf{y} \in L^2\left(\Omega_C ;\mathbb{R}^3\right),  \end{equation}
and $\A_s$ is the magnetic vector potential due to the assigned source current density $\mathbf{J}_S$. To be more precise,
\begin{equation}
\label{as}
    \A_s(\cdot,t) = \OACS \left[ \J_s(\cdot,t) \right],
\end{equation}
where the operator $\OACS$ is defined as
\begin{equation}\label{acs}
\OACS:\mathbf{v} \in H_L\left(\Omega_S\right)\rightarrow \frac{\mu_0}{4\pi}\int_{\Omega_S} \frac{\mathbf{v}(\mathbf{y})}{\|\mathbf{y}-\mathbf{x}\|}\mathrm{d}\mathbf{y} \in L^2\left(\Omega_C;\mathbb{R}^3\right).
\end{equation}

For a comprehensive discussion on the functional spaces associated with Maxwell's equations, we refer to~\cite{bossavit1998computational,bossavit1981numerical}.

To state and prove the next result, it is convenient to introduce an additional operator
\begin{equation}
\OASC:\mathbf{v} \in H_L\left(\Omega_C\right)\rightarrow \frac{\mu_0}{4\pi}\int_{\Omega_C} \frac{\mathbf{v}(\mathbf{y})}{\|\mathbf{y}-\mathbf{x}\|}\mathrm{d}\mathbf{y} \in L^2\left(\Omega_S;\mathbb{R}^3\right).
\end{equation}
We have the following Proposition whose demonstration is a straightforward consequence of the definitions of corresponding operators.
\begin{proposition}
\label{rmk:sym}
The operator $\OA_{sc}$ is the adjoint of $\OA_{cs}$, i.e.
\begin{equation}
    \langle  \OA_{sc} \ww_c , \ww_s \rangle_S = \langle \ww_c,  \OACS \ww_s \rangle_C,
    \ \forall \ww_c \in H_L \left( \Omega_C \right), \ww_s \in H_L \left( \Omega_S \right).
\end{equation}
\end{proposition}

\begin{remark}
\label{rmk:rmterm}
 Taking into account \eqref{as}, and with the help of \Cref{rmk:sym}, the second term on the r.h.s. of \eqref{governing_eq} can be written as 
\begin{equation}
    \langle \ww , \partial_t\A_s(\cdot,t) \rangle_C = \langle \OASC \left[ \ww \right], \partial_t\J_S(\cdot,t) \rangle_S.
\end{equation}
\end{remark}

\subsection{Uniqueness of Weak Solution}
In this section, the uniqueness of the weak solution of \eqref{governing_eq} is proved by using classical techniques. Let $\mathbf{J}_1,\;\mathbf{J}_2\in L^2\left(0,T;H_L\left( \Omega_C \right)\right)$ be two weak solutions of \eqref{governing_eq}. Then, $\mathbf{J}_0=\mathbf{J}_1 - \mathbf{J}_2$ in  $L^2\left(0,T;H_L\left( \Omega_C \right)\right)$ solves the variational problem 
\begin{equation}
\begin{split}
    \langle \ww,\eta \J_0(\cdot,t) \rangle_C  =  - \langle \ww, \partial_t \OACC \left[ \J_0\right(\cdot,t)] \rangle_C,
 \end{split}
 \label{governing_eq_hom}
\end{equation}
for every test function $\ww \in H_L\left(\Omega_C\right)$ and almost every $0\leq t \leq T$ with initial condition $\J_0(\mathbf{x},0)=0$ for almost every $\mathbf{x} \in \Omega_C$. Using arguments of \cite{tamburrino2021monotonicity}, it can be proved that if  $\mathbf{J}_0\left(\mathbf{x},t\right)\in L^2\left(0,T;H_L\left( \Omega_C \right)\right) $ is a solution of the homogeneous system \eqref{governing_eq_hom}, then it can be expressed as
\begin{equation}
\label{eq:sum_f_hom}
\mathbf{J}_0\left(\rr,t\right)  = \sum_{n=1}^\infty i_n\left(t\right)\ {\bf j}_n (\rr)\quad\textrm{in}\ \Omega_C \times [0,T],
\end{equation}
where $\jj_n$s are the normalized solutions of the following Generalized Eigenvalue Problem
\begin{equation}
\label{eq:modes}
    \langle \OACC \mathbf{j}_{n},\ww_c \rangle_C  =\tau_{n}(\eta) \langle \eta \mathbf{j}_{n},\ww_c \rangle_C, \ \forall \ww_c \in \HL{C},
\end{equation}
and $i_n$s are solutions of \begin{eqnarray}
\label{eq:ode_f2}
r_n i_n+l_n i_n'=0 \ \forall n \in \mathbb{N},
\end{eqnarray}
where $r_n=\left\langle \eta {\bf j}_n, {\bf j}_n\right\rangle_C$, $l_n=\left\langle \OACC {\bf j}_n,{\bf j}_n \right\rangle_C$ and time constants $\tau_n = l_n/r_n$ are finite, real, non-negative, ordered in decreasing order and converging to zero. Since $\J_0(\cdot,0)\equiv 0$ and $\{\mathbf{j}_{n}\}_{n\geq1}$ is a complete $\eta$-orthonormal set in $\HL{C}$, the initial condition $i_n (0)=0$ for all $n\geq1$ is obtained. Since $i_n(t)=i_n(0) e^{-t / \tau_n}$ from \eqref{eq:ode_f2}, it follows that $i_n\left(t\right) = 0$ for every $0\leq t \leq T$ and $n\geq 1$. This proves that $\mathbf{J}_0= 0$ for almost every $x \in \Omega_C$ and $0 \leq t \leq T$ and, hence, the uniqueness of the weak solution of the system \eqref{governing_eq}.

\subsection{Projection Operators}\label{pro_ope_subsec}
Since $H_L\left(\Omega_C \right)$ is a closed subspace of $L^2 \left( \Omega_C; \mathbb{R}^3 \right)$, its inner product naturally induces a projection operator
\begin{equation}
    \OP_C: \ww \in L^2 \left( \Omega_C; \mathbb{R}^3 \right) \rightarrow \ww_c \in H_L \left( \Omega_C \right),
\end{equation}
where $\ww_c = \argmin_{\uu \in H_L \left( \Omega_C \right)} \langle \ww - \uu , \ww - \uu \rangle_C$. It is worth noting that $\ww_c = \OP_C \ww$ is the unique element of $H_L \left( \Omega_C \right)$ such that
\begin{equation}
    \langle \ww, \uu \rangle_C = \langle \ww_c, \uu\rangle_C, \ \forall \uu \in H_L \left( \Omega_C \right).
\end{equation}
Similarly, one can define $\OP_S$, the projection operator over $\HL{S}$.

Therefore, \eqref{governing_eq} can be cast in the operator form
\begin{equation}
\label{GOV_PDE_Proj}
    \OP_C \left( \eta \OI + \OACC \partial_t \right) \OP_C \J  =  - \OP_C \OACS \partial_t \J_S,
\end{equation}
where $\OI$ is the identity operator, and \eqref{as} has been for the rightmost term.

The projection operator $\OP_C$ applied to $\J$ is \emph{not} strictly required but makes clear the symmetric structure of the operator on the l.h.s. of \eqref{governing_eq}.

\subsection{Measured Quantities}
Let $\E_R(\rr,t)$ be the reaction electric field, i.e., the electric field produced by the current density $\J$. The quantity measured is the projection of the reaction electric field $\E_R(\rr,t)$ on $\HL{S}$. We denote it by $\OP_S \E_R$.

Since $\E_R(\rr,t)=-\partial_t \A_R(\rr,t) - \nabla \varphi_R (\rr,t)$, where $\A_R=\OASC \J$ is the magnetic vector potential and $\varphi_R$ is the electric scalar potential (see \cite{haus1989electromagnetic}), by exploiting $\int_{\Omega_S} \nabla \varphi_R \cdot \J' \mathrm{d}V = 0$ for every $\J' \in \HL{S}$, it turns out that
\begin{align}
\label{eq:vectpot}
   \langle \E_R(\cdot,t), \ww_s \rangle_S & = -\langle \partial_t \A_R(\cdot,t), \ww_s \rangle_S\\
\label{eq:vectpot2}
    & = - \langle \partial_t \OASC \left[ \J(\cdot,t) \right], \ww_s \rangle_S
\end{align}
for every $\ww_s \in \HL{S}$, $\langle \cdot,\cdot \rangle_S$ representing the usual inner product defined on $L^2 \left( \Omega_S \right)$ .
As a consequence, it results
\begin{align}
    \OP_S \E_R & = - \OP_S \partial_t \A_R \\
\label{prj_er}
    & = - \OP_S \OASC \partial_t \J.
\end{align}
Equation \eqref{eq:vectpot} suggests another observable quantity, namely $\A_R$ or, more precisely, its projection $\OP_S \A_R$:
\begin{equation}
\label{prj_ar}
    \OP_S \A_R \\
     = \OP_S \OASC \J.
\end{equation}

\subsection{Structure of the Solution}In \cite{tamburrino2021monotonicity}, it has been shown that if  $\mathbf{J}\left(\rr,t\right)\in L^2\left(0,T;H_L\left( \Omega_C \right)\right) $ is a solution of \eqref{governing_eq}, then it can be expressed as:
\begin{equation}
\label{eq:sum_f}
\mathbf{J}\left(\rr,t\right)  = \sum_{n=1}^{+\infty} i_n\left(t\right)\ {\bf j}_n (\rr)\quad\textrm{in}\ \Omega_C \times [0,T],
\end{equation}
where the $\jj_n$s are the solutions of \eqref{eq:modes} and the $i_n$s are the solutions of \begin{eqnarray}
r_n i_n+l_n i_n'=\mathcal{E}_n \ \forall n \in \mathbb{N},
\end{eqnarray}
where $\mathcal{E}_n=-\langle\partial_{t}{\bf A}_S, {\bf j}_n\rangle_C$ and $\tau_n,$ $r_n,$ $l_n$ are the same as in \eqref{eq:ode_f2}.

By linearity and \eqref{eq:sum_f}, the solution of \eqref{governing_eq} can be written as
\begin{equation}
\label{eq:sol_split}
\mathbf{J}\left(\rr,t\right)  = \sum_{n=1}^{+\infty} c_n e^{-t/\tau_n} {\bf j}_n (\rr) + \J_F(\rr,t) \ \textrm{in}\ \Omega_C \times [0,T],
\end{equation}
where the series on the r.h.s. corresponds to the source-free response, while $\J_F$ is the forced response, i.e., a particular solution of \eqref{governing_eq} related to the specific prescribed source $\J_S$. The initial condition $\left( \J(\rr,t=0)) \text{ prescribed}\right)$ is imposed by the constants $c_n$s:
\begin{equation}
     c_m = \frac{\langle \J (\rr,0) - \J_F(\rr,0), \eta \jj_m \rangle_C}{\langle \jj_m, \eta \jj_m \rangle_C}, \ \forall m \in \mathbb{N}.
\end{equation}
\begin{proposition}
It is worth noting that
\begin{equation}
\label{eq:large_t}
 \left\lVert \sum_{n=1}^{+\infty} c_n e^{-t/\tau_n} {\bf j}_n (\rr) \right\rVert_C = \mathcal{O}\left(\mathrm{e}^{-t/\tau_1} \right), \textrm{ when } t \rightarrow + \infty.
\end{equation}
\end{proposition}
\begin{proof}
Indeed:
\[
\begin{split}
\left\lVert \sum_{n=1}^{+\infty} c_n e^{-t/\tau_n} {\bf j}_n (\rr) \right\rVert_C & \leq     \sum_{n=1}^{+\infty} |c_n| e^{-t/\tau_n}\left\lVert  {\bf j}_n (\rr) \right\rVert_C \\
&\leq \sum_{n=1}^{+\infty}\left|  \frac{\langle \J (\rr,0) - J_F(\rr,0), \eta \jj_n \rangle_C}{\langle \jj_n, \eta \jj_n \rangle_C} \right|e^{-t/\tau_n}  \left\lVert  {\bf j}_n (\rr) \right\rVert_C\\
&\leq\frac{1}{\eta_L\left\lVert  {\bf j}_n (\rr) \right\rVert_C}\sum_{n=1}^{+\infty}  \left|\langle \J (\rr,0) ,\eta \jj_n \rangle_C- \langle J_F(\rr,0), \eta \jj_n \rangle_C \right|e^{-t/\tau_n} \\
&\leq\frac{1}{\eta_L\left\lVert  {\bf j}_n (\rr) \right\rVert_C}\sum_{n=1}^{+\infty}(\eta_U  \lVert \J (\rr,0)\rVert_C \lVert \jj_n \rVert_C +\eta_U \lVert J_F(\rr,0)\rVert_C\lVert \jj_n \rVert_C) e^{-t/\tau_n} \\
&\leq \frac{\eta_U}{\eta_L} \left(\lVert \J (\rr,0)\rVert_C + \lVert J_F(\rr,0)\rVert_C\right) \sum_{n=1}^{+\infty} e^{-t/\tau_n},
\end{split}
\]
where $\eta_U$ and $\eta_L$ are the upper and lower bounds of $\eta$, respectively.
Let $\{\tau_{n_j}\}_{j\in\N}$ be a subsequence that contains all time constants without any repetition. Considering that the time constants are eigenvalues of \eqref{eq:modes} and that the operator $\mathcal A_{cc}$ is a linear operator, it is easily seen that the algebraic multiplicity is lower than the dimension $d$ of the host space for each connected component (see, e.g., \cite[Chap. XIII]{reed1978iv} and \cite[Chap. 1]{henrot2006extremum}). 
In the present case, the algebraic multiplicity of the eigenvalues is not larger than $3k$, since $d=3$ is the dimension of the host space, where $k$ is the number of connected components. Thus
\[
\sum_{n=1}^{+\infty} e^{-t/\tau_n}\leq 3k \sum_{j=1}^{+\infty} e^{-t/\tau_{n_j}}< + \infty,
\]
because the series to the right of the first inequality is convergent by the ratio test. Indeed, $e^{-t/\tau_{n_{j+1}}} / e^{-t/\tau_{n_j}} < 1$ since $\tau_{n_{j+1}}$ is strictly smaller than $\tau_{n_{j}}$.

Finally, it results that
\[
\sum_{n=1}^{+\infty} e^{-t/\tau_n} =  m_1 e^{-t/\tau_1}+e^{-t/\tau_1} \sum_{j=m_1+1}^{+\infty} e^{-t\left(1/\tau_{n}-1/\tau_{1}\right)},
\]
where $m_1$ is the multiplicity of $\tau_1$. The series on the r.h.s. is convergent for any $t>0$, as can be verified by using the ratio test. Moreover, since the series at the r.h.s. is monotonically decreasing, the convergence at an arbitrary $t>0$, gives the conclusion \eqref{eq:large_t}.
\end{proof}

\section{The Transfer Operator}
\label{transfer_sec}

This section is organized into three main parts.
In \Cref{subsec_exp}, the forced response is derived for problem \eqref{governing_eq} when the applied current density $\J_S$ is exponential in time, and the Transfer Operator is introduced. In \Cref{subsec:lap} the interpretation of the Transfer Operator is given in terms of the Laplace Transform. In \Cref{subsec:Inv} the existence of the Transfer Operator is proved.

\subsection{Exponential waveforms}\label{subsec_exp}
In this section, we analyze the existence of the forced response of the system \eqref{governing_eq}, when the applied current density $\J_S$ is exponential in time, i.e. we assume
\begin{equation}
\label{eq:expJS}
    \J_S(\rr,t) = \jj_S(\rr) e^{\lambda t},
\end{equation}
where $\lambda$ is a real number. 
 Combining \eqref{as}, \eqref{acs} and \eqref{eq:expJS}, we get
\begin{equation}\label{asexpo}
    \A_s(\mathbf{x},t) = e^{\lambda t}\OACS \left[ \jj_S(\cdot) \right] =e^{\lambda t} \frac{\mu_0 }{4\pi}\int_{\Omega_S} \frac{\jj_S(\mathbf{y})}{\|\mathbf{y}-\mathbf{x}\|}\mathrm{d}\mathbf{y},
\end{equation}
Due to the linearity of the system, the forced response $\J_F$ for the eddy current density is of the same exponential type
\begin{equation}
\label{eq:expJ}
    \J_F(\rr,t) = \jj(\rr) e^{\lambda t},
\end{equation}
where $\jj \in H_L\left(\Omega_C \right)$ solves the variational problem
\begin{equation}
\label{eq:GovWF_exp}
    \left \langle \ww_c,\left( \eta \OI  +\lambda \OACC \right) \jj \right \rangle_C   = 
    -\lambda \langle \ww_c , \OACS \jj_S \rangle_C, \ \forall
    \ww'_c \in H_L\left(\Omega_C \right). 
\end{equation}

The measured fields $\A_R$ and $\E_R$ are
\begin{align}
\label{eq:er}
    \A_R(\rr,t) & =\aaa_R(\rr) e^{\lambda t}
    \\
\label{eq:er2}
    \E_R(\rr,t) & =\ee_R(\rr) e^{\lambda t},
\end{align}
where, thanks to \eqref{prj_er} and \eqref{prj_ar}
\begin{align}
\label{eq:er1}
    \langle \ee_R, \ww_s \rangle_S & = -\lambda \langle  \OASC \jj, \ww_s \rangle_S,
    \\
\label{eq:ar}
    \langle  \aaa_R, \ww_s \rangle_S & =
    \langle  \OASC \jj, \ww_s \rangle_S.
\end{align}

In operator form \eqref{eq:GovWF_exp}, \eqref{eq:er1} and \eqref{eq:ar} are
\begin{align}
\label{GOV_exp_Proj}
    \OP_C \left( \eta \OI + \lambda \OACC  \right) \OP_C \jj &  =  -\lambda \OP_C \OACS \jj_S,
    \\
    \OP_S \ee_R & = -\lambda  \OP_S \OASC \jj,
    \\
    \OP_S \aaa_R & = \OP_S \OASC \jj.
\end{align}
Therefore, if the operator appearing on the l.h.s. of \eqref{GOV_exp_Proj} is invertible (see \Cref{subsec:Inv}), it turns out that
\begin{equation}
\label{eq:pser}
    \OP_S \ee_R = \OH_\eta \jj_S,
\end{equation}
where $\OH_\eta: H_L\left(\Omega_S \right) \to H_L\left(\Omega_S \right)$ is defined as follows
\begin{equation}
\label{eq:H_def}
    \OH_\eta = \lambda  \OP_S \OASC \left[ \OP_C \left( \eta \OI + \lambda \OACC  \right) \OP_C \right]^{-1} \lambda \OP_C \OACS.
\end{equation}

Combining \eqref{eq:sol_split}, \eqref{eq:large_t}, \eqref{eq:er2}, \eqref{eq:pser} and \eqref{eq:H_def}, it results that
\begin{equation}
\label{eq:ER_larget}
\OP_S \E_R  = \mathcal{O}\left(\mathrm{e}^{-t/\tau_1(\eta)} \right) + \OH_\eta \jj_S \, e^{\lambda t}, \textrm{ for } t \rightarrow + \infty.
\end{equation}
When $\lambda > -1/{\tau_1(\eta)}$, the second term on the r.h.s. of \eqref{eq:ER_larget} is dominant for large $t$ and can be experimentally measured.

\subsection{Laplace Transform}
\label{subsec:lap}
 Taking the Laplace transform of \eqref{governing_eq}, it results that
\begin{equation}
\label{eq:GovWF_Lap}
    \langle \ww'_c,\left( \eta \OI + p \OACC  \right) \J(\cdot,p) \rangle_C  = - p \langle  \ww'_c, \OACS  \left[ \J_S(\cdot,p) \right] \rangle_C, \ \forall
    \ww'_c \in H_L\left(\Omega\right),  
\end{equation}
where
\begin{equation}
\begin{split}
    \J(\cdot,p) & = \OL \left[ \J(\cdot,t) \right](p)\\
    \J_S(\cdot,p) & = \OL \left[ \J_S(\cdot,t) \right](p),\\ 
\end{split}
\end{equation}
being $\OL$ the Laplace transform operator, and $p$ the complex variable.
Similarly,
\begin{align}
\label{eq:ER_Lap}
    \langle \E_R(\cdot,p), \ww_s \rangle_S & = -p \langle  \OASC \left[ \J(\cdot,p) \right], \ww_s \rangle_S,
    \\
\label{eq:AR_Lap}
    \langle  \A_R(\cdot,p), \ww_s \rangle_S & = \langle  \OASC \left[ \J(\cdot,p) \right], \ww_s \rangle_S.
\end{align}

Equations \eqref{eq:GovWF_Lap}, \eqref{eq:ER_Lap}, and \eqref{eq:AR_Lap} are equal to \eqref{eq:GovWF_exp}, \eqref{eq:er1}
 and \eqref{eq:ar}, respectively, for $p=\lambda$. Therefore, the operator $\OH_\eta$ defined in \eqref{eq:H_def} is the Transfer Operator, i.e., the mapping $\J_S(\rr,p) \rightarrow \OP_S \E_R(\rr,p)$ evaluated at $p=\lambda$. Hereafter we will use the symbol $\OH_\eta$, for this operator, i.e.
\begin{equation}
\label{H_eta_def}
\OH_\eta: p\in \mathbb C \mapsto \OH_\eta(p) \in \mathfrak L (H_L\left(\Omega_S \right),H_L\left(\Omega_S) \right).
\end{equation}

\subsection{Existence of $\OH_\eta$}
\label{subsec:Inv}
In this Section, we prove that the solution of \eqref{eq:GovWF_exp} exists and is unique, if $\lambda$ is larger than $-1/\tau_1(\eta)$ where $\tau_1$ is the largest eigenvalue for problem \eqref{eq:modes}. The following results hold.
\begin{proposition}
\label{prop:eta_pos}
The operator\footnote{Hereafter, the same symbol $\eta$ is used to represent both a function $\eta(x)$ and the corresponding multiplication operator $\eta : f(x) \mapsto \eta(x) f(x)$.}
 $\OP_C \eta \OP_C$ is coercive for $\eta \in L_+^\infty \left( \Omega_C \right)$.
\end{proposition}
\begin{proof}
    By definition, if $\eta \in L_+^\infty \left( \Omega_C \right)$, there exists a constant $\eta_0>0$ such that $\eta(x) \geq \eta_0 \text{ for a.e. } x \in \Omega_C$. Then
    \begin{align}
        \langle \ww_c , \OP_C \eta \OP_C \ww_c \rangle_C
                & = \langle \OP_C \ww_c , \eta \OP_C \ww_c \rangle_C
        \\
        & \geq \langle \ww_c , \eta_0 \ww_c \rangle_C.
    \end{align}
\end{proof}

\begin{lemma}
\label{prop:op_coerc}
    For any $\eta \in L_+^\infty \left( \Omega_C \right)$, the operator $\OP_C \left( \eta \OI + \lambda \OACC  \right) \OP_C$ is coercive if $\lambda > -1/\tau_1(\eta)$.
\end{lemma}
\begin{proof}
First, we recall that (see (3.3) of \cite{tamburrino2021monotonicity})
\begin{equation}
 \tau_1(\eta) = \max_{\jj \in \HL{C} \backslash \{0\}} \frac{\langle \OACC \jj, \jj \rangle_C}{\langle \eta \jj , \jj \rangle_C}.   
\end{equation}
Therefore, for any $-1/\tau_1(\eta) < \lambda_0 < \min\{0,\lambda\}, $ it results that
\begin{align}
    \lambda \langle \OACC \jj, \jj \rangle_C \geq \lambda_0 \langle \OACC \jj, \jj \rangle_C \geq \lambda_0 \tau_1(\eta) \langle \eta \jj , \jj \rangle_C.
\end{align}
Consequently,
\begin{align}
\label{eq:pos_def}
    \left\langle \ww_c , \OP_C \left( \eta \OI + \lambda \OACC  \right) \OP_C \ww_c \right\rangle \geq \left\langle \ww_c , \eta \ww_c \right\rangle \left[ 1 + \lambda_0 \tau_1(\eta) \right].
\end{align}
Since the term in square brackets in \eqref{eq:pos_def} is a positive constant, the coercivity of $\OP_C \left( \eta \OI + \lambda \OACC  \right) \OP_C$ follows from \Cref{prop:eta_pos}.
\end{proof}

\begin{remark}
    It is possible to prove that the condition $\lambda > -1/\tau_1(\eta)$ is necessary for the coercivity of $\OP_C \left( \eta \OI + \lambda \OACC  \right) \OP_C$.
\end{remark}

\begin{lemma}
\label{prop:sf_bounded}
        For $\eta \in L_+^\infty \left( \Omega_C \right)$ and finite $\lambda$, the sesquilinear form $\OS : \HL{C} \times \HL{C} \mapsto \RR$ defined as:
        \begin{equation}
            \OS(\vv_c,\ww_c) = \langle \vv_c,\left( \eta \OI + \lambda \OACC \right) \ww_c \rangle_C       
        \end{equation}
        is bounded.
\end{lemma}
\begin{proof}
Indeed, it results that
    \begin{align}
        \left| \OS(\vv_c,\ww_c) \right| & \leq \left| \langle \vv_c, \eta \ww_c \rangle_C \right| + \left| \lambda \right| \left|\langle \vv_c, \OACC  \ww_c \rangle_C \right|
        \\
        & \leq \Vert \vv_c \Vert \, \Vert \eta \ww_c \Vert + \left| \lambda \right| \Vert \vv_c \Vert \Vert \OACC \Vert \, \Vert \ww_c \Vert
        \\
        \label{eq:bounded_sesq}
        & \leq \Vert \vv_c \Vert \, \Vert  \ww_c \Vert \left( \Vert \eta \Vert_{\infty}+ \left| \lambda \right| \Vert \OACC \Vert \right).
    \end{align}
Therefore, the sesquilinear form $\OS$ is bounded because (i) $\eta \in L_+^\infty \left( \Omega_C \right)$ and (ii) it is possible to prove that $\Vert \OACC \Vert = \Vert \eta \Vert_{\infty} \tau_1(\Vert \eta \Vert_{\infty})$.
\end{proof}
Furthermore, we have the following invertibility result.
\begin{proposition}
\label{thm:MainRes}
    For $\eta \in L_+^\infty \left( \Omega_C \right)$ and $\lambda > - 1/\tau_1(\eta)$, the operator $\OP_C \left( \eta \OI - \tau^{-1} \OACC  \right) \OP_C$ is invertible.
\end{proposition}
\begin{proof}
    The operator $\OP_C \left( \eta \OI + \lambda \OACC  \right) \OP_C$ is invertible if and only if the solution of equation $\OP_C \left( \eta \OI + \lambda \OACC  \right) \OP_C \vv  = \ww_c$, i.e. the solution of problem
\begin{equation}
\label{eq:Lax_Mil}
    \left \langle \ww'_c,\left( \eta \OI  + \lambda \OACC \right) \vv \right \rangle_C   = 
    - \lambda \langle \ww'_c , \ww_c \rangle_C, \     \forall
    \ww'_c \in H_L\left(\Omega_C \right), 
\end{equation}
exists and is unique $\forall \, \ww_c \in \HL{C}$. Equation \eqref{eq:Lax_Mil} admits a unique solution thanks to the Lax-Milgram Theorem that can be invoked because the operator $\OP_C \left( \eta \OI - \tau^{-1} \OACC  \right) \OP_C$ is coercive (\Cref{prop:op_coerc}) and the related sesquilinear form $\OS$ is bounded (\Cref{prop:sf_bounded}).
\end{proof}

We have the following key Theorem concerning the existence of $\OH_\eta$, resulting from Theorem \ref{thm:MainRes}.
\begin{theorem} 
\label{cor:MainStat}
    For $\eta \in L_+^\infty \left( \Omega_C \right)$ and $\lambda > -1 / \tau_1(\eta)$, the operator $\OH_\eta$ of \eqref{eq:H_def} exists and is well defined.
\end{theorem}

\section{Monotonicity of Transfer Function}
\label{sec:MonoTF}
In this section, it is proved that the operator $\eta \in L^{\infty}_+(\Omega_C) \rightarrow \OH_{\eta}(-1/\tau)$, mapping the material property $\eta$ to $\OH_{\eta}$, satisfies the Monotonicity Principle for $\tau$ negative or greater than a proper positive constant. The subscript $\eta$ in $\OH_{\eta}$ makes explicit the specific electrical resistivity on which it is evaluated \eqref{eq:H_def}.

We recall that:
\begin{itemize}
    \item[(a)] $\eta_1 \leq \eta_2$ means that
$\eta_1(x) \leq \eta_2(x) \text{ for a.e. } x \in \Omega_C$;
    \item[(b)] $\OA \preceq \OB$ means that $\OB - \OA$ is a positive semi-definite operator.
\end{itemize}

The main Theorem relies on the two following Lemmas.

\begin{lemma}
\label{lem:MonoEta}
    The mapping $\eta \in L^{\infty}\left(\Omega_C\right) \mapsto \OP_C \eta \OP_C$ is monotone, that is,
    \begin{equation}
        \eta_1 \leq \eta_2 \Longrightarrow \OP_C \eta_1 \OP_C \preceq \OP_C \eta_2 \OP_C.
    \end{equation}
\end{lemma}
\begin{proof}
    The proof is trivial. Indeed, if $\eta_1(x) \leq \eta_2(x) \text{ for a.e. } x \in \Omega_C$, then it follows that
    \begin{align}
        \left \langle \ww_c , \eta_1 \ww_c \right \rangle_C & = \int_{\Omega_C} \eta_1 \left| \ww_c(x) \right|^2 \text{d}V(x)
        \\
        & \leq \int_{\Omega_C} \eta_2 \left| \ww_c(x) \right|^2 \text{d}V(x)
        \\
        & = \left \langle \ww_c , \eta_2 \ww_c \right \rangle_C, \ \forall \ww_c \in \HL{C}.
    \end{align}
\end{proof}

\begin{lemma}
\label{lm:GenProp}
    Given two linear, invertible and symmetric operators $\OA_1$ and $\OA_2$ defined on the space $X$, an operator $\OC$ defined on the space $X$, and a linear operator $\OB$ defined from the space $Y$ to the space $X$, it turns out that:
    \begin{itemize}
        \item[(a)] $\OA_1 \preceq \OA_2 \Longleftrightarrow \OA_2^{-1} \succeq \OA_1^{-1}$;
        \item[(b)] $\OA_1 \preceq \OA_2 \Longleftrightarrow \OA_1 +\OC \preceq \OA_2 +\OC$;
        \item[(c)] $\OA_1 \preceq \OA_2 \Longrightarrow \OB^{\dag} \OA_1 \OB \preceq \OB^{\dag} \OA_2 \OB$.
    \end{itemize}
\end{lemma}
\begin{proof}
    The proofs are trivial.
\end{proof}

Hereafter, for the sake of simplicity, we define $\lambda_1(\eta)=-\tau_1^{-1}(\eta)$. It is worth noting that $\lambda_1(\eta)$ is a pole of the Transfer Operator $\OH_\eta$. In general, $-\tau_k^{-1}(\eta)$ is a pole for $\OH_\eta$, for arbitrary $k$.

The Monotonicity Principle for the Transfer Operator $\OH_\eta$ evaluated on the real axis is stated below.
\begin{theorem}\label{monotrans}
Let $\alpha,\beta \in L_+^\infty(\Omega_C)$ be two electrical resistivities, and let $\mathcal H_\alpha$ and $\mathcal H_\beta$ be defined as in \eqref{eq:H_def}, then
\begin{equation}
\label{eq:Monotonicity}
    \alpha \leq \beta \implies \OH_{\alpha} \left( \lambda \right) \preceq \OH_{\beta} \left( \lambda \right), \ \forall \lambda \in \mathbb{R} \textrm{ and } \lambda > \max\{\lambda_1(\alpha),\lambda_1(\beta)\}.
\end{equation}
\end{theorem}
\begin{proof}
First, the Transfer Operators $\OH_{\alpha} (\lambda)$ and $\OH_{\beta}(\lambda)$ are well defined for $\lambda > \max\{\lambda_1(\alpha),\lambda_1(\beta)\}$, thanks to \Cref{cor:MainStat}.
    Then $\alpha \leq \beta$ implies $\OP_C \alpha \OP_C \preceq \OP_C \beta \OP_C$, thanks to \Cref{lem:MonoEta}. Then, it results that
    \begin{align}
        \OP_C \alpha \OP_C \preceq \OP_C \beta \OP_C \Longleftrightarrow \ & \OP_C (\alpha \OI +\lambda \OACC ) \OP_C \preceq \OP_C (\beta \OI +\lambda \OACC ) \OP_C
        \\
        \label{eq:MonoInt}
        \Longleftrightarrow \ & \left[ \OP_C (\alpha \OI +\lambda \OACC ) \OP_C \right]^{-1} \succeq \left[ \OP_C (\beta \OI +\lambda \OACC ) \OP_C \right]^{-1},
    \end{align}
where we have exploited (a) and (b) from \Cref{lm:GenProp}. Finally, \eqref{eq:Monotonicity} follows from         \eqref{eq:MonoInt} combined with (c) from \Cref{lm:GenProp}, for $\OA_1 = \left[ \OP_C (\beta \OI +\lambda \OACC ) \OP_C \right]^{-1}$, $\OA_2 = \left[ \OP_C (\alpha \OI +\lambda \OACC ) \OP_C \right]^{-1}$, and $\OB=-\lambda \OP_C \OACS$.

\end{proof}

\section{A Non-Iterative Imaging Method Based on Monotonicity}
\label{method_sec}
This section shows how to convert the MP of \Cref{monotrans} into an imaging method. Conversions of this type (from an MP to an imaging method) have been known since the paper \cite{tamburrino2002new}. Hereafter, it is provided for the sake of completeness and with reference to an inverse obstacle problem consisting in the reconstruction of the shape of a conductive inclusion embedded in a conductive material $\Omega_C$. 

Let $A\subset\Omega_C$ be the region occupied by an inclusion with electrical resistivity $\eta_i$ embedded in a conductive background that occupies the region $\Omega_C$. The electrical conductivity of the background is indicated as $\eta_{BG}$. The overall electrical resistivity is 
\begin{equation}\label{incluresis}
    \eta_A (x)=
    \left\lbrace
    \begin{array}{ccc}
      \eta_i   & \hbox{ in } & A,\\
        \eta_{BG} & \hbox{ in } & \Omega_C\setminus A.
    \end{array}
    \right.
\end{equation}
Hereafter, it is assumed that
\begin{equation}\label{nin0relation}
\eta_i > \eta_0.
\end{equation}
It is easy to observe that if $A_1$ and $A_2$ are two possible inclusions in $\Omega_C$, then \eqref{incluresis} and \eqref{nin0relation} imply that
\begin{equation}\label{etanin0relation}
A_1\subset A_2 \implies \eta_{A_1}\leq \eta_{A_2}.
\end{equation}
\Cref{monotrans} together with equation \eqref{etanin0relation} provides the monotonicity of the Transfer Operator for the inverse obstacle problem, as started below.
\begin{theorem}\label{monotransinclu}
Let $A_1$ and $A_2$ be two inclusions embedded in the conducting material $\Omega_C$. Under assumptions \eqref{incluresis} and \eqref{nin0relation}, it follows that
\begin{equation}
\label{eq:monotransinclu}
A_1\subset A_2 \implies \OH_{A_1} \left( \lambda \right) \preceq \OH_{A_2} \left( \lambda \right), \ \forall  \lambda \in \mathbb{R} \textrm{ and } \lambda > \max\{\lambda_1(\eta_{A_1}),\lambda_1(\eta_{A_{2}})\}.
\end{equation}
\end{theorem}
Imaging the shape of an anomaly via the Monotonicity Principle can be achieved as originally proposed in \cite{tamburrino2002new}. Specifically, the following proposition holds.

\begin{proposition}
\label{monotransinclu:converse}
Let $A_1$ and $A_2$ be two inclusions embedded in the conducting material $\Omega_C$. Under assumptions \eqref{incluresis} and \eqref{nin0relation}. If there exists a real number $\lambda$ greater than $\max\{\lambda_1(\eta_{A_1}),\lambda_1(\eta_{A_1})\}$ such that $$\OH_{A_1} \left( \lambda \right) \not\preceq \OH_{A_2} \left( \lambda \right),$$ then  \Cref{monotransinclu} implies that
$$A_1 \not\subset A_2.$$
\end{proposition}
The concept of monotonicity in transfer functions allows the development of an imaging algorithm that does not require iterative calculations. Hereafter, the key steps of the method are summarized for completeness.

First, let $\Omega_C$ be covered by means of a set of \lq\lq small\rq\rq \ subdomains called test elements, i.e., $\Omega_C = \{T_j\}_{j=1}^N$. 
Second, check via \Cref{monotransinclu:converse} if a given test domain $T_j$ is suspected not to be completely included in the unknown anomaly $A \subset \Omega_C$. If there exists a $\lambda$ such that $\OH_{T_j} \left( \lambda \right) \not\preceq \OH_{A} \left( \lambda \right)$, then discard the test domain $T_j$.
Third, take as a reconstruction of $U$ the union of those test domains that have not been discarded during the second step.

Summing up, the reconstruction rule is
\begin{equation}
\label{eqn:r0}
        A^U=\bigcup_j \left\{T_j \ |\  \OH_{T_j}(\lambda) \preceq \OH_{A}(\lambda),\ \forall \lambda > \lambda_j \right\},
\end{equation}
where $\lambda_j=\max\{\lambda_1 \left(T_j \right), \lambda_1\left(A \right) \}$ is a proper threshold.

When the unknown anomaly $A$ is the union of some or all $T_j$s, then $A^U$ is a rigorous upper bound:
\begin{equation}
    A \subseteq A^U \subseteq A^U \left( \lambda^* \right),
\end{equation}
where $A^U \left( \lambda^* \right)$ is the estimated reconstruction at the prescribed  parameter $\lambda^*$, i.e.,
\begin{equation}
\label{eqn:UB_SO}
    A^U \left( \lambda^* \right)=\bigcup_j \left\{T_j \ |\  \OH_{T_j}(\lambda^*) \preceq \OH_{A}(\lambda^*)\right\}.
\end{equation}

There is another way to estimate $A$ using the same principle, but focusing on the search for the test elements that contain the true anomaly $A$. Specifically, this reconstruction rule is 
\begin{equation}
\label{eqn:LB}
        A^L=\bigcap_j \left\{T_j \ |\  \OH_{T_j}(\lambda) \succeq \OH_{A}(\lambda),\ \forall \lambda > \lambda_j \right\}.
\end{equation}
In other words, the final reconstruction of $A$ is obtained by finding the common regions between all the test elements that might contain $A$.

When the unknown anomaly $A$ is the intersection of some of the $T_j$s, then $A^L$ is a rigorous lower bound:
\begin{equation}
    A^L \left( \lambda^* \right) \subseteq A^L \subseteq A,
\end{equation}
where $A^L \left( \lambda^* \right)$ is the estimated reconstruction at the prescribed  parameter $\lambda^*$, i.e.,
\begin{equation}
\label{eqn:LB_SO}
    A^L \left( \lambda^* \right)=\bigcap_j \left\{T_j \ |\  \OH_{T_j}(\lambda^*) \succeq \OH_{A}(\lambda^*)\right\}.
\end{equation}

\section{Conclusions}
\label{conclusion_sec}
In this paper, the inverse obstacle problem is studied within the framework of Magnetic Induction Tomography, which involves the detection of small defects in a conducting material using low-frequency electromagnetic fields. From the mathematical point of view, this is a non-linear and ill-posed inverse problem, in which the underlying physics is governed by a parabolic PDE.

The proposed strategy is based on the search for a Monotonicity Property that, when available, can be translated into a non-iterative imaging method. In this setting (parabolic PDE), it is very challenging to recognize the proper operator showing the Monotonicity Property and to prove the related monotonicity. The main contribution of this study is that a Monotonicity Property for parabolic PDEs emerges when considering the Laplace transforms of a proper quantity that can be measured, the Transfer Operator $\mathcal{H}_\eta$. Specifically, the introduction of the Laplace transform, by hiding the temporal dependence of the problem, naturally gives rise to the Transfer Operator $\mathcal{H}_\eta$, which satisfies a proper Monotonicity Principle:
\begin{equation}
\label{eq:Monotonicity_conclusion}
    \alpha \leq \beta \implies \OH_{\alpha} \left( \lambda \right) \preceq \OH_{\beta} \left( \lambda \right). 
\end{equation}
Subsequently, an imaging method based on this Monotonicity Principle provides the solution for the inverse obstacle problem in real-time operations.

\section*{Acknowledgments}
This work has been partially supported by the Italian Ministry of University and Research (projects n. 2022Y53F3X and n. 20229M52AS, PRIN 2022), and by the National Group for Mathematical Analysis, Probability and their Applications (GNAMPA), of the Italian National Institute of Higher Mathematics (INdAM).

\section*{Authorship contribution statement}

{\bf A. Tamburrino}: Conceptualization, Funding acquisition, Methodology, Supervision, Validation, and Writing.

{\bf A. Corbo Esposito and G. Piscitelli}: Conceptualization, Funding acquisition, Methodology, Validation, and Writing. 

\bibliography{mybibfile}
\bibliographystyle{abbrv}
\end{document}